\documentclass[11pt]{article}

\usepackage{amsmath,amssymb,amsfonts,amsbsy,amsthm,mathrsfs}

\renewcommand{\subsection}{\subsubsection}

\newtheorem{proposition}{Proposition}[section]

\begin{document}

\title{\bf Structural stability of shock waves and current-vortex sheets in shallow water magnetohydrodynamics}

\author{{\bf Yuri Trakhinin}\\
Sobolev Institute of Mathematics, Koptyug av. 4, 630090 Novosibirsk, Russia\\
and\\
Novosibirsk State University, Pirogova str. 1, 630090 Novosibirsk, Russia\\
E-mail: trakhin@math.nsc.ru
}

\date{ }

\maketitle

\begin{abstract}
We study the structural stability of shock waves and current-vortex sheets in shallow water magnetohydrodynamics (SMHD) in the sense of the local-in-time existence and uniqueness of discontinuous solutions satisfying corresponding jump conditions. The equations of SMHD form a symmetric hyperbolic system which is formally analogous to the system of 2D compressible elastodynamics for particular nonphysical deformations. Using this analogy and the recent results in \cite{MTT19} for shock waves in 2D compressible elastodynamics, we prove that shock waves in SMHD are structurally stable if and only if the fluid height increases across the shock front. For current-vortex sheets the fluid height is continuous whereas the tangential components of the velocity and the magnetic field may have a jump. Applying a so-called secondary symmetrization of the symmetric system of SMHD equations, we find a condition sufficient for the structural stability of current-vortex sheets.
\end{abstract}

\vspace{2mm}
\noindent{\bf Keywords:} shallow water magnetohydrodynamics; symmetric hyperbolic system; local-in-time existence of discontinuous solutions; shock waves; current-vortex sheets.

\section{Introduction}
\label{s1}

The equations of shallow water magnetohydrodynamics (SMHD) were proposed by Gilman \cite{Gil} for studying the global dynamics of the solar tachocline which is a thing transition layer between the Sun's radiative interior and the differentially rotating outer convective zone. The SMHD equations are derived in \cite{Gil} from the equations of ideal incompressible magnetohydrodynamics (MHD) under the influence of gravity by depth averaging and assuming that the pressure is hydrostatic and
the depth of the layer of a perfectly conducting fluid is small enough. The SMHD system is important not only for astrophysical applications like the solar tachocline (see, e.g., \cite{DikGil,Gil,Zaq}) but may also be used for modelling conducting shallow water fluids in laboratory and industrial environments (further references about applications of SMHD can be found, for example, in \cite{JETP}).

The equations of SMHD \cite{Gil} read:
\begin{align}
& \frac{{\rm d} h}{{\rm d}t} +h{\rm div}\,{v} =0, \label{1}\\[3pt]
& \frac{{\rm d}v}{{\rm d}t}-({B}\cdot\nabla ){B}+g{\nabla}h  =0, \label{2}\\[3pt]
& \frac{{\rm d}B}{{\rm d}t}-(B\cdot\nabla )v=0\label{3},
\end{align}
where $h$ is the height of a conducting fluid, $v= (v_1,v_2)\in \mathbb{R}^2$ is the fluid velocity, $B=(B_1,B_2)\in \mathbb{R}^2$ is the magnetic field, the constant $g>0$ is the gravitational acceleration, ${\rm d} /{\rm d} t =\partial_t+({v} \cdot{\nabla} )$ is the material derivative, $\partial_t=\partial/\partial t$ is the time derivative, $\nabla =(\partial_1,\partial_2)$, $\partial_i=\partial/\partial x_i$ ($i=1,2$), $x=(x_1,x_2)$, and $x_1$ and $x_2$ are spatial coordinates.

System \eqref{1}--\eqref{3} is a closed system for the unknown $ U =U (t, x )=(h, v,B)\in \mathbb{R}^5$ and can be written as the following quasilinear system:
\begin{equation}
\label{3'}
A_0(U )\partial_tU+A_1(U)\partial_1U +A_2(U)\partial_2U=0,
\end{equation}
where $A_0= \mbox{block diag}\, (g/h ,I_4)$,
\[
A_1=\begin{pmatrix}
{\displaystyle\frac{gv_1}{h}} & ge_1 & \underline{0}  \\[7pt]
ge_1^{\top}& v_1I_2 & -B_{1}I_2  \\[3pt]
\underline{0}^{\top} &-B_{1}I_2 &  v_1I_2
\end{pmatrix},\quad
A_2=\begin{pmatrix}
{\displaystyle\frac{gv_2}{h}} & ge_2 & \underline{0}  \\[7pt]
ge_2^{\top}& v_2I_2 & -B_{2}I_2  \\[3pt]
\underline{0}^{\top} &-B_{2}I_2 &  v_2I_2
\end{pmatrix},
\]
$e_1=(1,0)$, $e_2=(0,1)$, $\underline{0}=(0,0)$ and  $I_m$ and $O_m$ denote the unit and zero matrices of order $m$ respectively. System \eqref{3'} is symmetric hyperbolic if   $A_0>0$, i.e., if the natural physical restriction
\begin{equation}
h >0 \label{3''}
\end{equation}
holds.

We do not include the equation
\begin{equation}\label{4}
{\rm div}(hB)=0
\end{equation}
into the full set \eqref{1}--\eqref{3} of the SMHD equations because \eqref{4} is just a divergence constraint on the initial data $U(0,x)=U_0(x)$. Indeed, from \eqref{1} and \eqref{3} we deduce
\begin{equation}\label{hB}
\partial_t(hB) + \nabla\times (hB\times v) + {\rm div} (hB)\,v=0
\end{equation}
that implies the linear equation ${\rm d} r/{\rm d}t+({\rm div}\, v)\,r=0$ for $r={\rm div}(h B)$. It follows from this equation that if \eqref{4} is satisfied initially, then it holds for all $t > 0$.

Using \eqref{4} and \eqref{hB}, we rewrite the SMHD equations \eqref{1}--\eqref{3} as the following system of conservation laws \cite{DeSt}:
\begin{equation}\label{5}
\left\{
\begin{array}{l}
 \partial_th +{\rm div} (h {v} )=0,\\[6pt]
 \partial_t(h {v} ) +{\rm div}(h{v}\otimes{v} -h{B}\otimes{B} ) +
{\nabla}(gh^2/2)=0, \\[6pt]
\partial_t(hB) + \nabla\times (hB\times v)=0.
\end{array}
\right.
\end{equation}
For system \eqref{5}, equation \eqref{4} is again a constraint on the initial data. Using \eqref{4}, we can derive from \eqref{5} equations \eqref{1}--\eqref{3}, i.e., systems  \eqref{1}--\eqref{3} and \eqref{5} are equivalent on smooth solutions (in fact, the same can be proved for piecewise smooth solutions with a regular discontinuity, e.g., a shock wave). In the next section we will write down the Rankine--Hugoniot jump conditions for the conservation laws \eqref{5} which should hold at each point of a regular strong discontinuity.

Some mathematical aspects for the hyperbolic system of SMHD like characteristics, Riemann invariants, simple wave solutions were studied by De Sterck \cite{DeSt}. He has also classified  discontinuous solutions satisfying the Rankine--Hugoniot  relations and discussed the Lax conditions for shock waves. The Riemann problem for the SMHD system was studied by Zaqarashvili et. al. \cite{Zaq}. We should also mention many studies of the stability of stationary solutions of the SMHD equations (see \cite{Hughes} and references therein).

As is known, strong discontinuities (e.g., shock waves) formally introduced for a system of hyperbolic conservation laws do not necessarily exist (at least, locally in time) as piecewise smooth solutions for the full range of admissible initial flow parameters. According to our know\-ledge, there were no studies of the local-in-time existence and uniqueness (structural stability) of discontinuous solutions of the SMHD system satisfying the corresponding Rankine--Hugoniot jump conditions. This is the main goal of the present paper.

Using a formal mathematical analogy between SMHD and 2D compressible elastodynamics \cite{Daf,Gurt,Jos} (see Section 3) as well as the recent results in \cite{MTT19} for shock waves in elastodynamics, we obtain a relatively instant result about the structural stability of SMHD shock waves under the fluid height increase assumption (see Section 4). We also consider current-vortex sheets in SMHD (see the next section). For them a crucial idea is the usage of a so-called secondary symmetrization proposed first in \cite{T05} for the system of ideal compressible MHD in the context of the study of the structural stability of classical current-vortex sheets. This enables us to find a condition sufficient for the structural stability of SMHD current-vortex sheets (see \eqref{ssc1} in Section 5). In the end of the paper, we also briefly discuss a possibility to find a necessary and sufficient structural stability condition for SMHD current-vortex sheets by using the recent results \cite{CHW1,CHW2,Hu} for compressible vortex sheets in 2D elastodynamics.

\section{Statement of the free boundary problems for shock waves and current-vortex sheets}
\label{s2}

Let $\Gamma (t)=\{ x_1=\varphi (t,x_2)\}$ be a curve of strong discontinuity for system \eqref{5}, i.e., we are interested in solutions of \eqref{5} that are smooth on either side of $\Gamma (t)$. As is known, to be weak solutions such piecewise smooth solutions should satisfy corresponding Rankine--Hugoniot jump conditions at each point of $\Gamma (t)$.  For the conservation laws \eqref{5} these jump conditions can be written in the following form:
\begin{align}
& [\mathfrak{m}]=0, \quad [\mathfrak{m}v_{\rm N}] +\frac{g}{2}|N|^2[h^2]=[h B_{\rm N}^{2}],\quad [\mathfrak{m}v_{\tau}]=[hB_{\rm N}B_{\tau}],\label{RH1}\\
& [\mathfrak{m}B_{\rm N}]=[h v_{\rm N}B_{\rm N}],\quad [\mathfrak{m}B_{\tau}]=[h v_{\tau}B_{\rm N}],\label{RH2}
\end{align}
where $[g]=g^+|_{\Gamma}-g^-|_{\Gamma}$ denotes the jump of $g$, with $g^{\pm}:=g$ in $\Omega^{\pm}(t)=\{\pm (x_1- \varphi (t,x_2))>0\}$, and
\[
\mathfrak{m}^{\pm}=h^{\pm} (v_{\rm N}^{\pm}-\partial_t\varphi),\quad v_{\rm N}^{\pm}=v^\pm\cdot N=v_1^{\pm}-v_2^{\pm}\partial_2\varphi ,\quad N=(1,-\partial_2\varphi),
\]
\[
B_{\rm N}^{\pm}=B^\pm\cdot N=B_{1}^{\pm}-B_{2}^{\pm}\partial_2\varphi ,\quad
v_{\tau}^{\pm}=v_1^{\pm}\partial_2\varphi +v_2^{\pm},\quad B_{\tau}^{\pm}=B_{1}^{\pm}\partial_2\varphi +B_{2}^{\pm}.
\]
Moreover, we have the condition $[h B_{\rm N}]=0$ coming from constraint \eqref{4}. On the other hand, the first condition in \eqref{RH2} is rewritten as $\partial_t\varphi [h B_{\rm N}]=0$. That is, this condition is implied by $[h B_{\rm N}]=0$ and can be thus excluded from \eqref{RH1}, \eqref{RH2}. Taking this and \eqref{3''} into account, we finally have the following system of jump conditions:
\begin{align}
& [\mathfrak{m}]=0, \quad [\mathfrak{b} ]=0,\quad [h]\big(\mathfrak{m}^2-\mathfrak{b}^2-\frac{g}{2}|N|^2\langle h\rangle h^+h^-|_{\Gamma}\big)=0,\label{RH1'} \\
& \mathfrak{m}[v_{\tau}]=\mathfrak{b} [B_{\tau}],\quad \mathfrak{m}[B_{\tau}]=\mathfrak{b} [v_{\tau}],\label{RH2'}
\end{align}
where $\mathfrak{m}:=\mathfrak{m}^{\pm}|_{\Gamma}$, $\mathfrak{b}:=h^{\pm} B_{\rm N}^{\pm}|_{\Gamma}$ and $\langle h\rangle =(h^++h^-)|_{\Gamma}$. The jump conditions \eqref{RH1'}, \eqref{RH2'} were written down in \cite{DeSt} for a stationary discontinuity (with $\varphi =0$).

Mathematically, all the regular strong discontinuities (satisfying the Rankine--Hugoniot conditions) are classified into two main types: {\it shock waves} and {\it characteristic discontinuities}. For shock waves, $\Gamma (t)$ should be not a characteristic of our quasilinear hyperbolic system. The curve $\Gamma (t)$ is a characteristic of system \eqref{3'} if $\det \widetilde{A}_1|_{\Gamma}=0$, where $\widetilde{A}_1$ is the so-called boundary matrix:
\begin{equation}
\widetilde{A}_1=\widetilde{A}_1(U,\varphi )= A_1(U)-A_0(U)\partial_t\varphi-A_2(U)\partial_2\varphi .\label{bmat}
\end{equation}
We easily calculate that
\begin{equation}\label{12}
\det \widetilde{A}_1(U^\pm,\varphi)|_{\Gamma}=\frac{g}{(h^{\pm})^6}\,\mathfrak{m}(\mathfrak{m}^2-\mathfrak{b}^2 )(\mathfrak{m}^2- \mathfrak{b}^2-g|N|^2(h^{\pm})^3)|_{\Gamma}.
\end{equation}

Let $[h]\ne 0$. Then, it follows from \eqref{12} and the last condition in \eqref{RH1'} that
\[
\det \widetilde{A}_1(U^\pm,\varphi)|_{\Gamma}=\left.\frac{\mp g^2[h](2\langle h\rangle -h^{\mp})}{2(h^{\pm})^5}\right|_{\Gamma}\mathfrak{m}(\mathfrak{m}^2-\mathfrak{b}^2 ).
\]
That is, for $[h]\ne 0$ we have $\det \widetilde{A}_1|_{\Gamma}=0$ if $\mathfrak{m}=0$ or $\mathfrak{m}^2=\mathfrak{b}^2$. If $[h]=0$ and $\mathfrak{m}^2\ne \mathfrak{b}^2$, then \eqref{RH1'},  \eqref{RH2'} imply $[v]=[B]=0$, i.e., the flow is continuous (for $\mathfrak{m}^2\ne \mathfrak{b}^2$, system \eqref{RH2'} implies $[v_{\tau}]=[B_{\tau}]=0$). Hence, we have shock waves if and only if $\mathfrak{m}\ne 0$, $\mathfrak{m}^2\ne \mathfrak{b}^2$ and $[h]\ne 0$. For shock waves, the jump conditions \eqref{RH1'},  \eqref{RH2'} are thus rewritten as
\begin{equation}
[\mathfrak{m}]=0, \quad [\mathfrak{b} ]=0,\quad \mathfrak{m}^2-\mathfrak{b}^2=\frac{g}{2}|N|^2\langle h\rangle h^+h^-,\quad
 [v_{\tau}]=0,\quad [B_{\tau}]=0\quad \mbox{on}\ \Gamma (t).\label{RH}
\end{equation}

The free boundary problem for shock waves is the problem for the systems
\begin{equation}
A_0(U^{\pm})\partial_tU^{\pm}+A_1(U^{\pm} )\partial_1U^{\pm}+A_2(U^{\pm} )\partial_2U^{\pm}=0\quad \mbox{in}\ \Omega^{\pm}(t),
\label{21}
\end{equation}
cf. \eqref{3'}, with the boundary conditions \eqref{RH} on $\Gamma (t)$ and  the initial
\begin{equation}
{U}^{\pm} (0,{x})={U}_0^{\pm}({x}),\quad {x}\in \Omega^{\pm} (0),\quad \varphi (0,{x}_2)=\varphi _0({x}_2),\quad {x}_2\in\mathbb{R}.\label{indat}
\end{equation}
Moreover, as for the Cauchy problem, the divergence constraint \eqref{4} is preserved by problem \eqref{RH}--\eqref{indat}.

\begin{proposition}
Suppose that problem \eqref{RH}--\eqref{indat} has a smooth solution  $(U^+,U^-,\varphi)$ for $t\in [0,T]$ satisfying the shock wave assumption $\mathfrak{m}\neq 0$. Then, if the initial data \eqref{indat} satisfy \eqref{4}, then
\begin{equation}\label{4'}
{\rm div}\,(h^{\pm} B^{\pm})=0\quad \mbox{in}\ \Omega^{\pm}(t)
\end{equation}
for all $t\in [0,T]$.
\label{p1}
\end{proposition}

\begin{proof}
Since $\mathfrak{m}\neq 0$, without loss of generality we may suppose that $v_{\rm N}^{\pm}|_{\Gamma}>\partial_t\varphi$.
It follows from \eqref{hB} that
\begin{equation}\label{hB'}
\partial_t(h^{\pm}B^{\pm}) + \nabla\times (h^{\pm}B^{\pm}\times v^{\pm}) + {\rm div} (h^{\pm} B^{\pm})\,v^{\pm}=0\quad \mbox{in}\ \Omega^{\pm}(t).
\end{equation}
Using then \eqref{hB'} and $v_{\rm N}^{\pm}|_{\Gamma}>\partial_t\varphi$ and following literally the arguments from \cite{MZ} towards the proof of the divergence constraint ${\rm div}\, H=0$ for the magnetic field $H$ on both side of the shock front in full MHD, we get constraints \eqref{4'} for all $t\in [0,T]$.
\end{proof}

We now consider a characteristic discontinuity for which $\mathfrak{m}=\mathfrak{b}=0$, i.e., $\partial_t\varphi=v_{\rm N}^{\pm}|_{\Gamma}$ and $B_{\rm N}^{\pm}|_{\Gamma}=0$. Then, the last condition in \eqref{RH1'} implies $[h]=0$ whereas from \eqref{RH2'} we see that $v_{\tau}$ and $B_{\tau}$ may have an arbitrary jump. This type of a characteristic discontinuity is a counterpart of tangential discontinuities (or current-vortex sheets) in full MHD \cite{BThand,LL,T05,T09}. By analogy with full MHD we will call such a discontinuity {\it current-vortex sheet}. We thus have the following boundary conditions on a curve of a current-vortex sheet:
\begin{equation}
\partial_t\varphi=v_{\rm N}^{\pm},\quad [h]=0\quad \mbox{on}\ \Gamma (t).\label{cvs}
\end{equation}
The free boundary problem for current-vortex sheets is the problem for systems \eqref{21} with the boundary conditions \eqref{cvs} and  the initial data \eqref{indat}. We do not include the conditions
\begin{equation}\label{bcon}
 B_{\rm N}^{\pm}=0\quad \mbox{on}\ \Gamma (t)
\end{equation}
into the boundary conditions \eqref{cvs} because, as for current-vortex sheets in full MHD \cite{T09}, they are just constraints on the initial data \eqref{indat}. The same is true for the divergence constraint \eqref{4} which is also preserved by problem \eqref{21}, \eqref{indat}, \eqref{cvs}.

\begin{proposition}
Suppose that problem \eqref{21}, \eqref{indat}, \eqref{cvs} has a smooth solution  $(U^+,U^-,\varphi)$ for $t\in [0,T]$. Then, if the initial data \eqref{indat} satisfy \eqref{4'} and \eqref{bcon} for $t=0$, then \eqref{4'} and \eqref{bcon} hold for all $t\in [0,T]$.
\label{p2}
\end{proposition}

The proof of Proposition \ref{p2} is totally analogous to the proof in \cite{T09} of corresponding divergence and boundary constraints for classical current-vortex sheets. To be exact, the proof in \cite{T09} was done in the terms of ``straightened'' variables, i.e., for a reduced problem in a fixed domain, but it clear that the assertion in \cite{T09} can be reformulated for the original free boundary problem.

In this paper, we are interested in shock waves and current-vortex sheets. Therefore, we just briefly discuss another possible discontinuous solutions.
The case $\mathfrak{m}=0$, $\mathfrak{b}\ne 0$ could be a counterpart of contact discontinuities in full MHD \cite{BThand,LL,MTT18}. However, for this case the boundary conditions \eqref{RH1'}, \eqref{RH2'}  together with the requirements $h^{\pm}>0$ yield $[h]=0$ and $[v]=[B]=0$, i.e., the flow is continuous. The remaining case for which we can have a characteristic discontinuity in SMHD is the case $\mathfrak{m}\ne 0$, $\mathfrak{m}^2=\mathfrak{b}^2$. Let, without loss of generality, $\mathfrak{m}=\mathfrak{b}$. Then, \eqref{RH1'}, \eqref{RH2'}  imply
\[
[h]=0,\quad [v_{\rm N}]=0,\quad [B_{\rm N}]=0,\quad [v_\tau ]=[B_\tau ],\quad v_{\rm N}^+ -\partial_t\varphi=B_{\rm N}^+.
\]
This type of characteristic discontinuity was called in \cite{DeSt} {\it Alfv\'{e}n discontinuity}. It is indeed a counterpart of classical Alfv\'{e}n discontinuities in full MHD \cite{BThand,LL,IT} because it propagates with the Alfv\'{e}n velocity $c_{a{\rm N}}=B_{\rm N}^{\pm}$ (in the normal direction).

\section{Formal analogy between SMHD and 2D compressible elastodynamics}

\label{s3}

By introducing the function $p=p(h)=(g/2)h^2$, we rewrite the SMHD system \eqref{1}--\eqref{3} as follows:
\begin{equation}
\left\{
\begin{array}{l}
{\displaystyle\frac{1}{hc^2}\,\frac{{\rm d} p}{{\rm d}t} +{\rm div}\,{v} =0,} \\[6pt]
{\displaystyle h\frac{{\rm d}v}{{\rm d}t}-h({B}\cdot\nabla ){B}+{\nabla}p  =0,} \\[6pt]
{\displaystyle h\frac{{\rm d}B}{{\rm d}t}-h(B\cdot\nabla )v=0,}
\end{array}\right. \label{1a}
\end{equation}
where $c=\sqrt{p'(h)}=\sqrt{gh}$ is the gravity wave speed (we will below assume that \eqref{3''} holds by default). We again have a symmetric hyperbolic system because equations \eqref{1a} are recast in form \eqref{3'} with
\begin{equation}
 {U} =(p,v,B),\quad {A}_0= \mbox{block diag}\, (1/(hc^2) ,hI_4),
\label{3'a}
\end{equation}
\begin{equation}
{A}_1=\begin{pmatrix}
{\displaystyle\frac{v_1}{hc^2}} & e_1 & \underline{0}  \\[7pt]
e_1^{\top}& hv_1I_2 & -hB_{1}I_2  \\[3pt]
\underline{0}^{\top} &-hB_{1}I_2 &  hv_1I_2
\end{pmatrix},\quad
{A}_2=\begin{pmatrix}
{\displaystyle\frac{v_2}{hc^2}} & e_2 & \underline{0}  \\[7pt]
e_2^{\top}& hv_2I_2 & -hB_{2}I_2  \\[3pt]
\underline{0}^{\top} &-hB_{2}I_2 &  hv_2I_2
\end{pmatrix}.
\label{3'b}
\end{equation}

We now consider the equations of elastodynamics \cite{Daf,Gurt,Jos} governing the  motion of compressible isentropic inviscid elastic materials.
For 2D flows, it is written as the following quasilinear symmetric system \cite{CHW1,CHW2,MTT19}:
\begin{equation}
\left\{
\begin{array}{l}
{\displaystyle\frac{1}{\rho c^2}\,\frac{{\rm d} p}{{\rm d}t} +{\rm div}\,{v} =0,}\\[6pt]
{\displaystyle\rho\, \frac{{\rm d}v}{{\rm d}t}+{\nabla}p -\rho (F_1\cdot\nabla )F_1 -\rho (F_2\cdot\nabla )F_2=0 ,}\\[6pt]
\rho\,{\displaystyle \frac{{\rm d} F_j}{{\rm d}t}-\rho \,(F_j\cdot\nabla )v =0,} \quad j=1,2,
\end{array}\right. \label{9}
\end{equation}
where $\rho$ is the density, $v\in\mathbb{R}^2$  is the velocity, $F_1=(F_{11},F_{21})$ and $F_2=(F_{12},F_{22})$ are the columns of the deformation gradient $F\in \mathbb{M}(2,2)$, the pressure $p=p(\rho)$ is a smooth function of $\rho$ and $c^2=p'(\rho)$ is the square of the sound speed. The corresponding symmetric matrices can be easily written down (see \cite{MTT19}). The symmetric system \eqref{9} for $U=(p,v,F_1,F_2)$ is hyperbolic under the natural conditions $\rho >0$ and $p'(\rho )>0$. Moreover, the divergence constraints ${\rm div}\,(\rho F_j)=0$ ($j=1,2$) hold at any time if they are satisfied initially.

We now can see a formal (mathematical) analogy between systems \eqref{1a} and \eqref{9}. Indeed, let us consider the equation of state $p = A\rho^2$ ($A > 0$) of a polytropic elastic medium with the adiabatic index $\gamma =2$. Then, assuming that $F_2\equiv 0$, dropping the last vector equation (for $j=2$) in \eqref{9} and using the formal notations $\rho :=h$, $F_1:=B$ and $A:=g/2$, we get system \eqref{1a}. Clearly, the assumption $F_2\equiv 0$ contradicting the physical requirement $\det F >0$ is nonphysical but we do not need to care about a physical sense of our formal analogy between SMHD and 2D compressible elastodynamics. Actually, if for system \eqref{9} we take initial data with $F_2|_{t=0}=0$, then we obtain $F_2=0$ for all $t>0$.

After rewriting equations \eqref{9} in the form of a system of conservation laws we can obtain for it jump conditions \cite{CHW1,MTT19} which for the case of shock waves read \cite{MTT19}:
\[
[\mathfrak{m}]=0,\quad (\mathfrak{m}^2-\mathfrak{b}^2)[1/\rho]+|N|^2[p]=0, \quad [v_{\tau}]=0,\quad [F_{j\tau}]=0,\quad [\rho F_{j{\rm N}}]=0,\quad j=1,2,
\]
where $\mathfrak{m}=\mathfrak{m}^{\pm}|_{\Gamma}=\rho^\pm (v_{\rm N}^{\pm}-\partial_t\varphi)|_{\Gamma}$, $\mathfrak{b}=(\rho^\pm )^2(F_{1{\rm N}}^\pm)^2+(F_{2{\rm N}}^\pm)^2)|_{\Gamma}$, $F_{j{\rm N}}^{\pm}=F_{1j}^{\pm}-F_{2j}^{\pm}\partial_2\varphi $, etc. For the equation of state $p = A\rho^2$, setting $F_2\equiv 0$, $\rho :=h$, $F_1:=B$ and $A:=g/2$, we get the jump conditions \eqref{RH}.

\section{Structural stability of SMHD shock waves}

\label{s4}

We can reduce the free boundary problem  \eqref{RH}--\eqref{indat} (with $U$ and the matrices given by \eqref{3'a}, \eqref{3'b}) to that in the fixed domains $\mathbb{R}^2_{\pm}=\{\pm x_1>0,\ x_2\in\mathbb{R}\}$ by the simple change of variables ${x}'_1=x_1-\varphi (t,x_2)$. Dropping primes, from systems \eqref{21} we obtain
\begin{equation}
A_0(U^{\pm})\partial_tU^{\pm}+\widetilde{A}_{1}(U^{\pm},\varphi )\partial_1U^{\pm}+A_2(U^{\pm} )\partial_2U^{\pm}=0\quad \mbox{for}\ x\in \mathbb{R}^2_{\pm},
\label{23}
\end{equation}
where $\tilde{A}_{1}$ was defined in \eqref{bmat}. The boundary conditions for \eqref{23} are \eqref{RH} on the line $x_1=0$.
It is well-known that the {\it necessary} condition for the structural stability of shock waves is the fulfilment of the Lax $k$--shock conditions \cite{Lax}
\[
\lambda_{k-1}^- <\partial_t\varphi <\lambda_k^-,\quad \lambda_k^+ <\partial_t\varphi <\lambda_{k+1}^+
\]
for some fixed integer $k$, where for our case of system of five equations $1\leq k\leq 5$ and $\lambda_j^\pm$ ($j=\overline{1,5}$) are the eigenvalues of the matrices $A_{\rm N}^{\pm}:=\left(A_0(U^{\pm})\right)^{-1}\left( A_1(U^{\pm})-A_2(U^{\pm})\partial_2\varphi  \right)|_{x_1=0}$. Moreover, $\lambda_j^\pm$ are numbered as
\[
\lambda_1^-\leq \ldots \leq\lambda_5^-,\quad \lambda_1^+\leq \ldots \leq\lambda_5^+,
\]
and we take $\lambda^-_0:=-2|\partial_t\varphi |$ and $\lambda_6^+:=2|\partial_t\varphi |$. The the Lax conditions just define a correct number of boundary conditions which is in agreement with the number of outgoing (from the boundary to the domain) characteristics, i.e., with the number of positive/negative eigenvalues of the matrices $\left(A_0(U^{\pm})\right)^{-1}\widetilde{A}_{1}(U^{\pm},\varphi )|_{x_1=0}$.

By direct calculation we find the eigenvalues $\lambda_j^\pm$ (cf. \cite{DeSt}):
\[
 \lambda_{1,5}^{\pm}=v_{\rm N}^{\pm}\mp c_{g{\rm N}}^\pm,\quad \lambda_{2,4}^{\pm}=v_{\rm N}^{\pm}\mp c_{a{\rm N}}^\pm ,\quad  \lambda_3^{\pm}=v_{\rm N}^{\pm}\quad \mbox{at}\ x_1=0,
\]
where
\[
c_{g{\rm N}}^\pm =\sqrt{(B_{\rm N}^{\pm})^2+gh^{\pm}|N|^2},\quad c_{a{\rm N}}^\pm =B_{\rm N}^{\pm},
\]
and we assume for definiteness that $B_{\rm N}^{\pm}|_{x_1=0}\geq 0$ and $v_{\rm N}^{\pm}|_{x_1=0}>\partial_t\varphi$ (i.e., $\mathfrak{m}>0$ and $\mathfrak{b}>0$). Then, $k$--shocks with $k\geq 3$ are not realizable. By virtue of the first two boundary conditions in \eqref{RH}, the inequalities $\lambda_2^->\partial_t\varphi$ and $\lambda_2^+<\partial_t\varphi$ appearing for $k=2$ contradict each other. That is, as in gas dynamics (as well as in elastodynamics \cite{MTT19}), only 1-shocks are possible:
\begin{equation}
v_{\rm N}^--\partial_t\varphi >c_{g{\rm N}}^-, \quad c_{a{\rm N}}^+<v_N^+-\partial_t\varphi <c_{g{\rm N}}^+ \quad \mbox{at}\ x_1=0.
\label{25}
\end{equation}

We note that 1--shocks are {\it extreme shocks} in the sense that ahead of the shock there are no waves outgoing from the shock  ($\lambda_j^->\partial_t\varphi $ for all $j=\overline{1,5}$). We see that the Lax conditions \eqref{25} guarantee the fulfilment of the shock wave assumptions $\mathfrak{m} \neq 0$ and $\mathfrak{m}^2\neq\mathfrak{b}^2$. Moreover, multiplying the first and third inequalities by $h^-_{|x_1=0}$ and $h^+_{|x_1=0}$ respectively and taking into account the first two boundary conditions in \eqref{RH}, we get
\[
\sqrt{\mathfrak{b}^2+g(h^{-})^3|N|^2}<\mathfrak{m}<\sqrt{\mathfrak{b}^2+g(h^{+})^3|N|^2}.
\]
This implies $[h]>0$.  On the other hand, using the third condition in \eqref{RH}, we can easily show that $[h]>0$ implies the fulfilment of the Lax conditions \eqref{25}.
That is, for shock waves the Lax conditions \eqref{25} are equivalent to the sole inequality $[h]>0$. Since the Lax conditions are necessary for the structural stability of shock waves, we conclude that {\it the fluid height increases} across the front of a structurally stable shock.

The Lax shock conditions is a 1D criterium whose fulfilment is necessary for multidimensional structural stability, but they are, in general, not sufficient for the local-in-time existence of corresponding shock front solutions. However, for SMHD shock waves we can show that the requirement $[h]>0$ being satisfied at each point of the initial shock front is also sufficient for structural stability (by default we, of course, assume that the initial data satisfy corresponding regularity conditions, compatibility conditions, etc.). As is known \cite{BThand,Maj,Met,MZ}, the most important (and, in some sense, sufficient) step towards the proof of structural stability is the proof of uniform stability of corresponding planar (rectilinear for the 2D case) shock waves. Uniform stability means the fulfilment of the uniform Kreiss--Lopatinski condition \cite{Kreiss,Maj,Met} by the linearized constant coefficient problem associated with a planar (or rectilinear for the 2D case) shock wave.

Since the SMHD equations are Galileo invariant, without loss of generality we can consider the unperturbed rectilinear shock wave with the equation $x_1=0$. We consider a constant solution
$(U^+,U^-,\varphi) =(\widehat{U}^+,\widehat{U}^-,\hat{\varphi})$ of systems \eqref{23} and the boundary conditions \eqref{RH} associated with this shock wave:
\[
\widehat{U}^{\pm}=(\hat{p}^{\pm},\hat{v}^{\pm},\widehat{B}^{\pm})=(g(\hat{h}^{\pm})^2/2,\hat{v}_1^{\pm},\hat{v}_2^{\pm},\widehat{B}_1^{\pm},\widehat{B}^{\pm}_2), \quad \hat{\varphi} =0,\quad  \hat{c}^{\pm}=\sqrt{g\hat{h}^{\pm}}.
\]
Here and below all the hat values are given constants. In view of the third condition in \eqref{RH}, $\hat{v}_2^+=\hat{v}_2^-$ and we can choose a reference frame in which $\hat{v}_2^+=\hat{v}_2^-=0$. The rest constants satisfy the relations
\begin{equation}
\frac{\hat{h}^+}{\hat{h}^-}=\frac{\hat{v}_1^-}{\hat{v}_1^+}=\frac{\widehat{B}_1^-}{\widehat{B}_1^+},\quad (\hat{v}_1^+)^2-(\widehat{B}_{1}^+)^2=\frac{g\hat{h}^-}{2}\Big( 1 + \frac{\hat{h}^-}{\hat{h}^+}\Big),\quad \widehat{B}_{2}^+=\widehat{B}_{2}^-
\label{sbc}
\end{equation}
following from \eqref{RH}. We assume that $\hat{v}_1^{\pm}>0$ and $\widehat{B}_1^{\pm}>0$. The fluid height increase assumption $\hat{h}^+>\hat{h}^-$ guarantees that the Lax conditions \eqref{25} hold for the constant solution:
\begin{equation}
M_1<M<M_*,
\label{Mach}
\end{equation}
where $M={\hat{v}_1^+}/{\hat{c}^+}$ is the downstream Froude number (here we mimic the notations from \cite{MTT19} where $M$ denotes the Mach number), $M_1=\widehat{B}_1^+/{\hat{c}^+}$ and $M_*=\sqrt{1+M_1^2}$.

Linearizing problem \eqref{23}, \eqref{RH} (at $x_1=0$), \eqref{indat} about the described constant solution, we get a linear constant coefficient problem for the perturbations $\delta U^{\pm}$ and $\delta\varphi$. Since our shock waves are extreme shocks, all the characteristics of the linear system in $\mathbb{R}^2_-$
for the perturbation $\delta U^-$ ahead of the shock are incoming in the shock, i.e., this linear system does not need any boundary conditions. As is known, without loss of generality we may then assume that $\delta U^-= 0$. Behind of the shock we obtain the linear constant coefficients system
\begin{equation}
A_0(\widehat{U}^+)\partial_t(\delta U^+)+\widetilde{A}_{1}(\widehat{U}^+,0 )
\partial_1(\delta U^+)
+A_2(\widehat{U}^+ )\partial_2(\delta U^+)=0\quad \mbox{for}\ x\in \mathbb{R}^2_+
\label{ls}
\end{equation}
for the perturbation $\delta U^+=(\delta p^+,\delta v^+,\delta B^+)$, where $\delta v^+=(\delta v^+_1,\delta v^+_2)$, etc.

As in \cite{MTT19}, it is convenient to reduce \eqref{ls} to a dimensionless form by introducing the following scaled values:
\[
x'=\frac{x}{l},\quad t'=\frac{\hat{v}_1^+t}{l},\quad p=\frac{\delta p^+}{\hat{h}^+(\hat{c}^+)^2},\quad v=\frac{\delta v^+}{\hat{v}_1^+},\quad B=\frac{\delta B^+}{\hat{c}^+},\quad  \varphi =\frac{\delta\varphi}{l},
\]
where $l$ is a typical length (recall that $\hat{v}_1^+>0$). The reduction to a dimensionless form is the same as in \cite{MTT19} if we formally set $\hat{h}^+=\hat{\rho}^+$, $\delta B^+=\delta F_1$ (see \eqref{9}), etc. After dropping the primes and taking into account that $\hat{v}_2^+=0$, system \eqref{ls} is rewritten as the following linear system for the scaled perturbation $U=(p,v,B)$ (with $v=(v_1,v_2)$ and $B=(B_1,B_2)$) behind of the shock wave:
\begin{equation}
\left\{
\begin{array}{ll}
Lp+{\rm div}\,v=0, & \\ M^2Lv-(\mathcal{B}\cdot\nabla )B+\nabla p =0, & \\ LB- (\mathcal{B}\cdot\nabla )v=0 & \quad \mbox{for}\ x\in \mathbb{R}^2_+,
\end{array}\right. \label{lsm}
\end{equation}
where $L =\partial_t+\partial_1$, $\mathcal{B}=(M_{1},{M}_{2})$ and $M_2=\widehat{B}_2^+/{\hat{c}^+}$. If in the corresponding linear system behind of the shock wave in 2D elastodynamics \cite{MTT19} we formally set $\mathcal{F}_2=F_2=0$, where $\mathcal{F}_2=\widehat{F}_2^+/\hat{c}^+$ (see \cite{MTT19}), and then drop the last ``$0=0$'' vector equation (for $F_2$), then we get our system \eqref{lsm}.

Taking into account \eqref{sbc}, $\delta U^-= 0$ and omitting technical calculations, we get the following linearized boundary conditions for system \eqref{lsm} written in a dimensionless form (in fact, we can avoid technical calculations if in the corresponding boundary conditions in \cite{MTT19} we just set $\mathcal{F}_2 =0$, $F_2=0$, etc.):
\begin{equation}\label{lbc}
\left\{
\begin{array}{ll}
{\displaystyle v_1 +d_0p -\frac{\ell_0}{M^2R}\,v_2=0}, &\\[9pt]
 a_0p+(1-R)\partial_{\star}\varphi =0,\qquad
v_2 +(1-R)\partial_2\varphi= 0, & \\[9pt]
{\displaystyle B_{1}+M_{1}\,p-\frac{M_{2}}{R}\,v_2=0,} \qquad B_{2}-M_{1}\,v_2=0 &\quad \mbox{at}\ x_1=0,
\end{array}
\right.
\end{equation}
where
\[
d_0=\frac{M_*^2+M^2}{2M^2},\quad R=\frac{\hat{h}^+}{\hat{h}^-},\quad \ell_0=M_1M_2,\quad  a_0=-\frac{\beta^2R}{2M^2},
\]
\[
\beta =\sqrt{M_*^2-M^2}\quad (\mbox{cf.}\ \eqref{Mach}),\quad \partial_{\star}=\partial_t-\frac{\ell_0}{M^2}\,\partial_2.
\]
Here we again mimic all the notations from \cite{MTT19}. In view of our fluid height increase assumption, we have $R>1$.  Our constant coefficients linearized problem is thus \eqref{lsm}, \eqref{lbc} with the initial data
\begin{equation}
{U} (0,{x})={U}_0({x}),\quad {x}\in \mathbb{R}^2,\quad \varphi (0,{x}_2)=\varphi _0({x}_2),\quad {x}_2\in\mathbb{R}.\label{lindat}
\end{equation}
As in \cite{MTT19}, we can show that the linearized version of a nonlinear divergence constraint is just a restriction on the initial data. In our case, the linearization of \eqref{4'} written a dimensionless form is
\begin{equation}\label{8l}
{\rm div}\,B +(\mathcal{B}\cdot\nabla p)=0\quad \mbox{for}\ x \in\mathbb{R}_2^+ .
\end{equation}
If the initial data \eqref{lindat} satisfy \eqref{8l}, then \eqref{8l} holds for all $t\in [0,T]$.

By adapting the energy method proposed by Blokhin \cite{Bl79} (see also \cite{BThand}) for gas dynamical shock waves, it was proved in \cite{MTT19} that all compressive shock waves in 2D elastodynamics are uniformly stable for convex equations of state. We could literally follow arguments in \cite{MTT19} (technical calculations for SMHD shock are even much simpler as in elastodynamics), but, in fact, we do not need to do this because all the arguments in \cite{MTT19} stay valid for the nonphysical case when $\det \widehat{F}^+ =0$ for the unperturbed deformation gradient $\widehat{F}^+$ behind of the shock. In particular, we can set $\widehat{F}_2^+=0$. Then, for the initial data with $F_2|_{t=0}=0$, where $F_2$ is the perturbation of the second column of the deformation gradient, we get $F_2=0$ for all $t>0$. Since the function $p=p(h)=(g/2)h^2$ in \eqref{1a} is convex, after that we can just translate the results in \cite{MTT19} to our problem \eqref{lsm}--\eqref{lindat} and conclude that all the rectilinear SMHD shock waves with $[\hat{h}]>0$ are uniformly stable.

Moreover, the condition $[\hat{h}]>0$ being equivalent to the Lax condition is necessary and sufficient for uniform stability and, referring to \cite{MTT19}, we obtain the a priori estimate
\begin{equation}
\begin{split}
\| U\|_{H^2([0,T]\times\mathbb{R}^2_+)}+ &
\| U_{|x_1=0}\|_{H^2([0,T]\times\mathbb{R})} + \| \varphi\|_{H^3([0,T]\times\mathbb{R})}
\\ &
\leq C\big\{\| U_0\|_{H^2(\mathbb{R}^2_+)} +\| \varphi_0\|_{L^2(\mathbb{R})} +\| f\|_{H^2([0,T]\times\mathbb{R}^2_+)}
+\| g\|_{H^2([0,T]\times\mathbb{R})}\big\},
\end{split}
\label{aprest'}
\end{equation}
for problem \eqref{lsm}--\eqref{lindat} with a given source term $f(t,x)\in \mathbb{R}^5$ in the right-hand side of the interior equations \eqref{lsm} and a given source term $g(t,x_2)\in \mathbb{R}^5$ in the right-hand side of the boundary conditions \eqref{lbc}. We refer the reader to \cite{MTT19} for the whole details of deriving estimate \eqref{aprest'}. Here we just note that the energy method proposed in \cite{Bl79} and adapted in \cite{MTT19} for shock waves in elastodynamics is based on the usage of a symmetrization of the wave equation for $p$. For our problem, using the boundary constraint \eqref{8l}, from system \eqref{lsm} we easily derive the hyperbolic equation
\[
\left(M^2{L}^2- \Delta -(\mathcal{B}\cdot\nabla )^2\right)p=0\quad \mbox{for}\ x \in\mathbb{R}_2^+
\]
whose canonical form is the wave equation. The crucial point is that, unlike full MHD \cite{BThand}, we can obtain a separate equation for $p$.

Since the priori estimate \eqref{aprest'} is derived in \cite{MTT19} by the construction of a so-called strictly dissipative 2-symmetrizer \cite{Tsiam}, referring to \cite{Tsiam}, we get the structural (nonlinear) stability of SMHD shock waves for which the fluid height increase condition $[h]>0$ holds at each point of the initial shock front.

\section{Structural stability of current-vortex sheets in a shallow layer of a conducting fluid}

\label{s5}

We now consider the free boundary problem \eqref{21}, \eqref{indat}, \eqref{cvs} for current-vortex sheets in SMHD. For current-vortex sheets in ideal compressible MHD the crucial idea was the secondary symmetrization of the MHD system proposed in \cite{T05}. The usage of this symmetrization gives a condition sufficient for the neutral stability of a planar current-vortex sheet, i.e., a condition sufficient for the fulfillment of the Kreiss--Lopatinski condition for the constant coefficients linearized problem. Having in hand this condition and assuming that it holds at each point of the initial current-vortex sheet, the local-in-time existence and uniqueness theorem for the nonlinear problem was independently proved in \cite{CW1,CW2} and \cite{T09}.

A secondary symmetrization of the quasilinear symmetric system \eqref{1}--\eqref{3}/\eqref{3'} can be found in a much easier way than for full MHD. Indeed, let $\lambda =\lambda (U)$ be an arbitrary function of the unknown $U$. From system  \eqref{1}--\eqref{3} we easily deduce the equations
\begin{align}
\frac{g}{h}\left\{\partial_th + ((v-\lambda B)\cdot \nabla )h\right\} +g{\rm div}\,{v} - g \lambda{\rm div}\,B & =0, \label{1s}\\[3pt]
 \partial_tv-\lambda\partial_tB +((v+\lambda B)\cdot \nabla )v-((B+\lambda v)\cdot \nabla )B+g{\nabla}h & =0, \label{2s}\\[3pt]
 -\lambda\partial_tv+\partial_tB -((B+\lambda v)\cdot \nabla )v+((v+\lambda B)\cdot \nabla )B-g\lambda{\nabla}h & =0\label{3s}
\end{align}
which can be written as the following symmetric system:
\begin{equation}
\label{3's}
B_0(U )\partial_tU+B_1(U)\partial_1U +B_2(U)\partial_2U=0,
\end{equation}
where
\[
B_0=\begin{pmatrix}
{\displaystyle\frac{g}{h}} & \underline{0} &\underline{0} \\[7pt]
\underline{0}^{\top}& I_2& -\lambda I_2 \\[3pt]
\underline{0}^{\top} & -\lambda I_2 & I_2
\end{pmatrix},\quad
B_1=\begin{pmatrix}
{\displaystyle\frac{g}{h}(v_1-\lambda B_1)} & ge_1 & -g\lambda e_1  \\[7pt]
ge_1^{\top}& (v_1+\lambda B_1)I_2 & -(B_{1}+\lambda v_1)I_2  \\[3pt]
-g\lambda e_1^{\top}&-(B_{1}+\lambda v_1)I_2 &  (v_1+\lambda B_1)I_2
\end{pmatrix},
\]
\[
B_2=\begin{pmatrix}
{\displaystyle\frac{g}{h}(v_2-\lambda B_2)} & ge_2 & -g\lambda e_2  \\[7pt]
ge_2^{\top}& (v_2+\lambda B_2)I_2 & -(B_{2}+\lambda v_2)I_2  \\[3pt]
-g\lambda e_2^{\top}&-(B_{2}+\lambda v_2)I_2 &  (v_2+\lambda B_2)I_2
\end{pmatrix}.
\]
System \eqref{3's} is hyperbolic ($B_0>0$) under the natural physical restriction \eqref{3''} and the condition
\begin{equation}\label{3''s}
  |\lambda |<1.
\end{equation}
Equations \eqref{2s} and \eqref{3s} are just linear combinations of \eqref{2} and \eqref{3} whereas \eqref{1s} comes from \eqref{1} and the divergence constraint \eqref{4} written as $(B\cdot \nabla )h+h{\rm div}\,B=0$.

The linearized constant coefficient problem associated with a rectilinear current-vortex sheet $x_1=0$ reads:
\begin{align}
 B_0(\widehat{U}^\pm )\partial_tU^\pm +B_1(\widehat{U}^\pm )\partial_1U^\pm +B_2(\widehat{U}^\pm )\partial_2U^\pm =0 & \qquad \mbox{for }x\in\mathbb{R}^3_\pm, \label{lcvs1}\\[6pt]
 \partial_t\varphi=v_1^{\pm}-\hat{v}_2^\pm\partial_2\varphi,\quad [h]=0 & \qquad \mbox{at}\ x_1=0,\label{lcvs2}
\end{align}
\begin{equation}
{U}^\pm (0,{x})={U}^\pm_0({x}),\quad {x}\in \mathbb{R}^2_{\pm},\quad \varphi (0,{x}_2)=\varphi _0({x}_2),  \quad {x}_2\in\mathbb{R},\label{lcvs3}
\end{equation}
where $\widehat{U}^\pm =(\hat{h},0,\hat{v}_2^\pm , 0,\hat{B}_2^\pm )$, $\hat{\varphi}=0$ is a constant solution of \eqref{21}, \eqref{cvs}, \eqref{bcon} whereas $U^\pm =(h^\pm,v^\pm,B^\pm)$ and $\varphi$ denote perturbations. The boundary conditions \eqref{lcvs2} are the linearization of \eqref{cvs} about this constant solution. Moreover, the conditions
\begin{equation}\label{lcvs4}
B_1^{\pm}=\widehat{B}_2^\pm\partial_2\varphi \quad \mbox{at}\ x_1=0
\end{equation}
being the linearizations of \eqref{bcon} are boundary constraints on the initial data \eqref{lcvs3}.

For systems \eqref{lcvs1} we obtain the energy identity
\[
\frac{{\rm d}I(t)}{{\rm d}t} =\int\limits_{\mathbb{R}}\big[(B_1(\widehat{U})U\cdot U )\big]{\rm d}x_2,
\]
where
\[
I(t)= \int\limits_{\mathbb{R}^2_+}(B_0(\widehat{U}^+)U^+\cdot U^+ ){\rm d}x + \int\limits_{\mathbb{R}^2_-}(B_0(\widehat{U}^-)U^-\cdot U^- ){\rm d}x
\]
and, in view  of  \eqref{lcvs2} and \eqref{lcvs4}, we have
\[
\big[(B_1(\widehat{U})U\cdot U )\big]=2g[h(v_1-\hat{\lambda}B_1)]=h^+_{|x_1=0}[\hat{v}_2-\hat{\lambda}\widehat{B}_2]\partial_2\varphi ,
\]
with $\hat{\lambda}^{\pm}=\lambda (\widehat{U}^\pm )$. Assuming that $\widehat{B}_2^+\neq 0$ or $\widehat{B}_2^-\neq 0$  and choosing such $\hat{\lambda}^{\pm}$ that $[\hat{v}_2-\hat{\lambda}\widehat{B}_2]=0$, we come to $I(t)=I(0)$. Under the hyperbolicity condition \eqref{3''s} satisfied for $\hat{\lambda}^{\pm}$, we then deduce an $L^2$ a priori estimate for the perturbations $U^{\pm}$ (see \cite{T05}). We can also derive a priori estimates for $U^{\pm}$ and $\varphi$ with loss of derivatives from the source terms introduced in the right-hand sides of \eqref{lcvs1} and \eqref{lcvs2} as well as analogous a priori estimates for the case of variable coefficients (see \cite{T05,T09}).

For writing down an exact form of a condition sufficient for the linear (neutral) stability of a rectilinear current-vortex sheet we have to analyze the requirements
\begin{equation}\label{lcvs5}
[\hat{v}_2]=\hat{\lambda}^+\widehat{B}_2^+ -\hat\lambda^-\widehat{B}_2^- ,\quad |\hat\lambda ^{\pm}|<1
\end{equation}
(for $\widehat{B}_2^+\neq 0$ or $\widehat{B}_2^-\neq 0$) coming from  $[\hat{v}_2-\hat{\lambda}\widehat{B}_2]=0$ and \eqref{3''s}. The widest range of the parameters $\hat{v}_2^\pm$ and $\widehat{B}_2^\pm$ is achieved for such a choice of $\hat{\lambda}^\pm$ that $|\hat\lambda^+|=|\hat\lambda^-|=[\hat{v}_2]/(|\widehat{B}_2^+|+|\widehat{B}_2^-|)$. Then, the inequalities in \eqref{lcvs5} give us the desired sufficient stability condition
\begin{equation}\label{ssc}
|[\hat{v}_2]|< |\widehat{B}^+_2|+|\widehat{B}^-_2|.
\end{equation}

The rest arguments are the same as in \cite{T09} and we conclude the structural stability of current-vortex sheets provided that the {\it sufficient stability condition} (cf. \eqref{ssc})
\begin{equation}
\left\{
\begin{array}{l}
 |B^+_2|_{\Gamma}+|B^-_2|_{\Gamma}-|[v_2]|\geq \varepsilon >0\\
 \mbox{and}\\
 |B^+_2|_{\Gamma}\geq \varepsilon >0\quad \mbox{or}\quad |B^-_2|_{\Gamma}\geq \varepsilon >0
 \end{array}
\label{ssc1}
\right.
\end{equation}
holds for the initial data \eqref{indat} with some  fixed constant $\varepsilon$ (of course, we should also assume that the initial data satisfy the hyperbolicity condition \eqref{3''} as well as appropriate regularity and compatibility conditions, see \cite{T09}). In spite of the fact that the first inequality in \eqref{ssc1} implies $|B^+_2|_{\Gamma}+|B^-_2|_{\Gamma}\neq 0$, we should additionally assume that either $B_2^+$ or $B_2^-$ does not vanish at each point of $\Gamma$. This assumption is necessary for resolving \eqref{cvs}, \eqref{bcon} for $(\partial_t\varphi ,\partial_2\varphi )$ in the same way at each point of $\Gamma$ and applying the arguments from \cite{T05,T09}.

In principle, for finding a condition which is not only sufficient but also necessary for structural stability we can again use the formal analogy between SMHD and 2D compressible elastodynamics. To this end one needs to revisit the study in \cite{CHW1,CHW2,Hu} for compressible vortex sheets in 2D elastodynamics for adapting it for SMHD current-vortex sheets. The nonlinear existence theorem for vortex sheets was proved in \cite{Hu} for a finite (not necessarily short) time but
under the condition that the initial discontinuity is close to a linearly stable rectilinear discontinuity. In fact,  the assumption that the initial data are
close to a piecewise constant solution could be removed provided that the time of existence is sufficiently short.

We now just translate the results from \cite{CHW1,CHW2,Hu} to particular SMHD current-vortex sheets which are close to a rectilinear discontinuity. Let us consider  a piecewise constant solution
\[
(\rho ,v,F_1,F_2) =(\hat{\rho} , 0,\hat{v}_2^\pm ,0,\widehat{F}_{21}^\pm,0,\widehat{F}_{22}^\pm)\quad \mbox{for } \pm x_1>0
\]
of system \eqref{9} associated with a rectilinear vortex sheet with the equation $x_1=0$. In \cite{CHW1,CHW2,Hu}, without loss of generality a reference frame in which $\hat{v}_2^+=-\hat{v}_2^-$ and, in view of the physical sense of the deformation gradient, a scale of measurement for which $\widehat{F}_{21}^+=-\widehat{F}_{21}^-$ and $\widehat{F}_{22}^+=-\widehat{F}_{22}^-$ were chosen (here and below we translate the notations from \cite{CHW1,CHW2,Hu} to ours). For SMHD, it seems it is impossible to change scale of measurement such that $\widehat{B}^+_2=-\widehat{B}^-_2$. This is why a rectilinear current-vortex sheet for which $\widehat{B}^+_2=-\widehat{B}^-_2$ is a particular case.

Using the formal analogy between SMHD and elastodynamics described above (in particular, we formally set $\widehat{F}_{22}^+=0$), the results obtained in \cite{CHW1,Hu} and the fact that for the chosen reference frame $\hat{v}_2^+=[\hat{v}_2]/2$, we obtain the following necessary and sufficient linear stability condition for a rectilinear current-vortex sheet with $\widehat{B}^+_2=-\widehat{B}^-_2$:
\begin{equation}\label{nsc}
|[\hat{v}_2]| \leq 2|\widehat{B}_2^+| \quad \mbox{or}\quad |[\hat{v}_2]| \geq 2\sqrt{(\widehat{B}_2^+)^2+2g\hat{h}}.
\end{equation}
For the particular case $\widehat{B}^+_2=-\widehat{B}^-_2$ our sufficient linear stability condition \eqref{ssc} reads: $|[\hat{v}_2]|< 2|\widehat{B}^+_2|$. That is, up to the transition point $|[\hat{v}_2]|= 2|\widehat{B}^+_2|$, for which one has a weaker a priori estimate in \cite{CHW1}, our sufficient stability condition successfully covers one of the two ``alternative'' parts of the whole stability domain described by \eqref{nsc}.

Moreover, referring to \cite{Hu} for a final nonlinear existence result for vortex sheets and translating it to SMHD current-vortex sheets, we can conclude the existence of current-vortex sheets being a small perturbation of a rectilinear current-vortex sheet with $\widehat{B}^+_2=-\widehat{B}^-_2$, provided that
\begin{equation}\label{nsc'}
|[\hat{v}_2]| < 2|\widehat{B}_2^+| \quad \mbox{or}\quad |[\hat{v}_2]| > 2\sqrt{(\widehat{B}_2^+)^2+2g\hat{h}}
\end{equation}
and
\begin{equation}\label{nsc''}
\begin{split}
& |[\hat{v}_2]|\neq |\widehat{B}_2^+|,\quad  |[\hat{v}_2]| \neq \sqrt{(\widehat{B}_2^+)^2+g\hat{h}}-|\widehat{B}_2^+| ,\\[6pt]
& |[\hat{v}_2]|\neq \sqrt{(\widehat{B}_2^+)^2+g\hat{h}},\quad |[\hat{v}_2]|\neq |\widehat{B}_2^+|\sqrt{\frac{(\widehat{B}_2^+)^2+2g\hat{h}}{(\widehat{B}_2^+)^2+g\hat{h}}}.
\end{split}
\end{equation}
It is natural that for the nonlinear result one needs strict inequalities in \eqref{nsc'}. At the same time, as is noted in \cite{CHW1,CHW2,Hu}, the exceptional points like those in \eqref{nsc''} have no physical meaning and appear due to certain requirements of the technique of Kreiss-type symmetrizers and the paradifferential calculus (an analogous nonphysical restriction appears in \cite{MT,MTW} for vortex sheets for the 2D nonisentropic Euler equations).
Indeed, the exceptional point $|[\hat{v}_2]|= |\widehat{B}_2^+|$ in any case appears inside of the stability subdomain $|[\hat{v}_2]| < 2|\widehat{B}_2^+|$, but the energy method giving us the sufficient stability $|[\hat{v}_2]| < 2|\widehat{B}_2^+|$ (for $\widehat{B}^+_2=-\widehat{B}^-_2$, cf. \eqref{ssc})  does not require that $|[\hat{v}_2]|\neq |\widehat{B}_2^+|$.

At last, we note that in this paper we do not describe mathematical details of the formulated linear and nonlinear stability/existence results like functional setting, regularity assumptions, compatibility conditions, etc. For such issues we refer the reader to the works cited above. In particular, here we do not describe how can one improve the structural stability result for SMHD current-vortex sheets (under the sufficient stability condition \eqref{ssc1}) in the sense that instead of the usage of the anisotropic weighted Sobolev spaces $H^s_*$ (as in \cite{T05,T09}) we can derive a priori estimates for the linearized problem and formulate the final nonlinear result in the classical Sobolev spaces $H^s$. We only note here that, as for current-vortex sheets in incompressible MHD \cite{MTT08}, this can be done by estimating missing normal derivatives through a current-vorticity-type linearized system.  In fact, the analogous idea was used in \cite{T18} for the free boundary problem in compressible elastodynamics where the role of the current is played by $\nabla\times F_j$ ($j=1,2,3$), with $F_j$ being columns of the deformation gradient $F\in \mathbb{M}(3,3)$.

\section*{Acknowledgements}
This work was supported in part by RFBR (Russian Foundation for Basic Research) grant No. 19-01-00261-a.


\begin{thebibliography}{100}

\bibitem{Bl79} Blokhin A.M. The mixed problem for the system of equations of acoustics with boundary conditions on a shock wave. {\it Izv. Sibirsk. Otdel. Akad. Nauk SSSR Ser. Tekhn. Nauk} {\bf 13} (1979), 25--33 (in Russian).

\bibitem{BThand}
Blokhin A., Trakhinin Y.  Stability of strong discontinuities in fluids and MHD. In: {\it Handbook of mathematical fluid dynamics},  Friedlander S., Serre D. (eds.), {\bf 1}, North-Holland, Amsterdam, 2002, pp. 545--652.

\bibitem{CW1}
Chen G.-Q. Wang Y.-G. Existence and stability of compressible current-vortex sheets in three-dimensional magnetohydrodynamics. {\it Arch. Ration. Mech. Anal.} {\bf 187} (2008), 369-408.

\bibitem{CW2}
Chen G.-Q. Wang Y.-G. Characteristic discontinuities and free foundary froblems for hyperbolic conservation laws. In: {\it Nonlinear Partial Differential Equations. The Abel Symposium 2010}, Holden H., Karlsen K.H. (eds.), Springer Verlag, Heidelberg, 2012, pp. 53--81.

\bibitem{CHW1}
Chen R.M., Hu J., Wang D. Linear stability of compressible vortex sheets in two-dimensional elastodynamics. {\it Adv. Math.} {\bf 311}  (2017), 18–60

\bibitem{CHW2}
Chen R.M., Hu J., Wang D. Linear stability of compressible vortex sheets in 2D elastodynamics: variable coefficients. {\it Math. Ann.} (2019), https://doi.org/10.1007/s00208-018-01798-w.

\bibitem{Daf}
Dafermos C.M. {\it Hyperbolic Conservation Laws in Continuum Physics}. 4th ed. Grundlehren Math. Wiss., vol. 325. Springer-Verlag, Berlin, 2016.

\bibitem{DikGil}
Dikpati M., Gilman P.A. Prolateness of the solar tachocline inferred from latitudinal force balance in a magnetohydrodynamic shallow-water model. {\it Astrophys. J.} {\bf 552} (2001), 348--353.

\bibitem{DeSt}
De Sterck H. Hyperbolic theory of the ``shallow water'' magnetohydodynamics equations. {\it Phys. Plasmas} {\bf 8} (2001), 3293--3304.

\bibitem{Gil}
Gilman P.A. Magnetohydrodynamic ``shallow water'' equations for the solar tachocline. {\it Astrophys. J. Lett.} {\bf 544} (2000), L79--L82.

\bibitem{Gurt}
Gurtin M.E. {\it An introduction to Continuum Mechanics}. Mathematics in Science and Engineering, vol. 158. Academic Press, New York--London, 1981.

\bibitem{Hu}
Hu J. {\it Vortex sheets in elastic fluids}. PhD thesis, University of Pittsburgh, 2017.

\bibitem{IT}
Ilin K.I., Trakhinin Y.L. On the stability of Alfv\'{e}n discontinuity. {\it Phys. Plasmas} {\bf 13} (2006), 102101--102108.

\bibitem{Jos}
Joseph D.  {\it Fluid Dynamics of Viscoelastic Liquids}. Applied Mathematical Sciences, vol. 84. Springer-Verlag, New York, 1990.

\bibitem{JETP}
Karelsky K.V., Petrosyan A.S., Tarasevich S.V. Nonlinear dynamics of magnetohydrodynamic flows of a heavy fluid on slope in the shallow water approximation. {\it J. Exp. Theor. Phys.} {\bf 119} (2014), 311--325.

\bibitem{Kreiss}
Kreiss H.-O. Initial boundary value problems for hyperbolic systems. {\it Commun. Pure Appl. Math.} {\bf 23} (1970), 277--296.

\bibitem{LL}
Landau L.D., Lifshiz E.M.,  Pitaevskii, L.P. {\it Electrodynamics of continuous media}.  Pergamon Press, Oxford, 1984.

\bibitem{Lax}
Lax P.D. Hyperbolic systems of conservation laws. II. {\it Commun. Pure and Appl. Math.} {\bf 10} (1957), 537--566.

\bibitem{Maj}
Majda A. {\it Compressible Fluid Flow and Systems of Conservation Laws in Several Space Variables}. Springer-Verlag,  New York, 1984.

\bibitem{Hughes}
Mak J., Griffiths S.D., Hughes D.W. Shear flow instabilities in shallow-water magnetohydrodynamics. {\it J. Fluid Mech.} {\bf 788} (2016), 767--796.

\bibitem{Met}
M\'etivier  G.  Stability of multidimensional shocks. In:
{\it Advances in the theory of shock waves}, H. Freist\"uhler, A. Szepessy (eds.), Progr. Nonlinear Differential Equations Appl. {\bf 47}, Birkh\"auser, Boston, 2001, pp. 25--103.

\bibitem{MZ}
M\'etivier G., Zumbrun K, Hyperbolic boundary value problems for symmetric systems with variable multiplicities. {\it J. Differential Equations} {\bf 211}  (2005), 61--134.

\bibitem{MTT08}
Morando A., Trakhinin Y., Trebeschi P. Stability of incompressible current-vortex sheets. {\it J. Math. Anal. Appl.} {\bf 347} (2008), 502--520.

\bibitem{MTT18}
Morando A., Trakhinin Y., Trebeschi P. Local existence of MHD contact discontinuities. {\it  Arch. Ration. Mech. Anal.} {\bf 228} (2018), 691--742.

\bibitem{MTT19}
Morando A., Trakhinin Y., Trebeschi P. Structural stability of shock waves in 2D compressible elastodynamics. {\it Math. Ann.} (2019), https://doi.org/10.1007/s00208-019-01920-6.

\bibitem{MT}
Morando A., Trebeschi P. Two-dimensional vortex sheets for the nonisentropic euler equations: Linear stability. {\it J. Hyperbolic Differ. Equ.} \textbf{5} (2008), 487--518.

\bibitem{MTW}
Morando A., Trebeschi P., Wang T. Two-dimensional vortex sheets for the nonisentropic Euler equations: Nonlinear stability. {\it J. Differential Equations} {\bf 266}  (2019), 5397--5430.


\bibitem{T05}
Trakhinin Y. On existence of compressible current-vortex
sheets: variable coefficients linear analysis. {\it Arch. Ration. Mech. Anal.} {\bf 177} (2005) 331--366.

\bibitem{Tsiam} Trakhinin Y. Dissipative symmetrizers of hyperbolic problems and their applications to shock waves and
characteristic discontinuities. {\it SIAM J. Math. Anal.} {\bf 37}  (2006), 1988--2024.

\bibitem{T09}
Trakhinin Y. The existence of current-vortex sheets in ideal compressible magnetohydrodynamics. \textit{Arch. Ration. Mech. Anal.} \textbf{191} (2009), 245--310.

\bibitem{T18}
Trakhinin Y. Well-posedness of the free boundary problem in compressible elastodynamics. {\it J. Differential Equations} {\bf 264}  (2018), 1661--1715.

\bibitem{Zaq}
Zaqarashvili T.V., Oliver R., Ballester J.L. Global shallow water magnetohydrodynamics waves in the solar tachocline. {\it Astrophys. J. Lett.} {\bf 691} (2009), L41--L44.


\end{thebibliography}
\end{document}